\documentclass[12pt]{article} 
\usepackage{amsmath,amssymb,amsthm,latexsym,color,graphicx,enumerate}
\usepackage{hyperref} 
\usepackage{fullpage}
\usepackage{float}
\def\nz{\ifmmode {I\hskip -3pt N} \else {\hbox {$I\hskip -3pt N$}}\fi}
\def\zz{\ifmmode {Z\hskip -4.8pt Z} \else
       {\hbox {$Z\hskip -4.8pt Z$}}\fi}
\def\qz{\ifmmode {Q\hskip -5.0pt\vrule height6.0pt depth 0pt
       \hskip 6pt} \'else {\hbox
       {$Q\hskip -5.0pt\vrule height6.0pt depth 0pt\hskip 6pt$}}\fi}
\def\rz{\ifmmode {I\hskip -3pt R} \else {\hbox {$I\hskip -3pt R$}}\fi}
\def\cz{\ifmmode {C\hskip -4.8pt\vrule height5.8pt\hskip 6.3pt} \else
       {\hbox {$C\hskip -4.8pt\vrule height5.8pt\hskip 6.3pt$}}\fi}

\def\qed{\hbox {\hskip 1pt \vrule width 4pt height 6pt depth 1.5pt
        \hskip 1pt}\\}
\newcommand {\pa}{\partial}

\makeatletter
\@addtoreset{equation}{section}

\makeatother

\newtheorem{theorem}{Theorem}[section]
\newtheorem{lemma}[theorem]{Lemma}
\newtheorem{proposition}[theorem]{Proposition}

\newtheorem{remark}[theorem]{Remark}
\newtheorem{corollary}[theorem]{Corollary}

\newtheorem{conjecture}[theorem]{Conjecture}

\title{Eigenvalues of the Neumann magnetic Laplacian in the unit disk}
\date{\today}
\author{Bernard Helffer\footnote{Laboratoire de Math\'ematiques Jean Leray, Nantes Universit\'e and CNRS, Nantes, France.
{\tt Bernard.Helffer@univ-nantes.fr}}, Corentin L\'ena\footnote{Universit\`a degli Studi di Padova, Dipartimento di Tecnica e Gestione dei Sistemi Industriali (DTG), Stradella S. Nicola 3, 36100 Vicenza and  Dipartimento di Matematica ``Tullio Levi-Civita", via Trieste 63, 35121 Padova, Italy. {\tt corentin.lena@unipd.it}. Corresponding author}}

\begin{document}
\bibliographystyle{plain}
\maketitle
\begin{abstract} 
In this paper, we study the first eigenvalue of the magnetic Laplacian with Neumann boundary conditions in the unit disk $\mathbb D$ in $\mathbb R^2$. There is a rather complete asymptotic analysis when the constant magnetic field tends to $+\infty$  and some inequalities seem to hold for any value of this magnetic field, leading to rather simple conjectures. Our goal is to explore these questions by revisiting a classical picture of the physicist D. Saint-James theoretically and numerically. On the way, we revisit the asymptotic analysis  in light of  the asymptotics obtained by Fournais-Helffer, that we can improve by combining them with  a formula stated by Saint-James.

\end{abstract}
\section{Introduction}
\label{s:intro}

In this paper, we study the first eigenvalue of the magnetic Laplacian with Neumann boundary conditions in the unit disk $\mathbb D$ in $\mathbb R^2$. Let us formulate the problem more precisely. Using the standard choice 
\[\mathbf A:x=(x_1,x_2)\in\mathbb R^2\mapsto (-x_2/2,x_1/2)\in\mathbb R^2\]
for the so-called magnetic potential, we define the self-adjoint operator $ H_\beta$ (with $\beta\in\mathbb R$) by 
\[{ H_\beta u}:=(i\nabla+\beta\,\mathbf A(x))^2u=-\Delta u+2i\,\beta\,\mathbf A\cdot \nabla u+\beta^2|\mathbf A|^2u\,,\]
for all $u$ in the domain
\[\mathcal D:=\{u\in H^2(\mathbb D)\,:\,\nabla u\cdot \vec{\nu}=0\mbox{ on }\partial \mathbb D\},\]
where $\vec{\nu}$ is the outward-pointing unit normal. It is well known that $H_\beta$ is a self-adjoint operator in $L^2(\mathbb D)$, non-negative (positive if $\beta\neq0$),  with compact resolvent. Let $\lambda(\beta)$ denote its lowest eigenvalue. 
Our goal is to understand in detail the dependence of $\lambda(\beta)$ on the parameter $\beta$ (interpreted as the strength of the magnetic field).\\  
The analysis of this Neumann problem is strongly motivated by the analysis of surface  superconductivity in physics and the Neumann boundary condition is crucial (see \cite{SdG}). At the mathematical level the case of the disk is a way towards understanding the role of the curvature of the boundary. We refer to \cite{BS}, \cite{He-Mo} and to \cite{FH4} for a presentation of the state of the art in 2014 including (to mention some of them)  works by Bernoff-Sternberg \cite{BS}, Bauman-Phillips-Tang \cite{BPT} (where the analysis on the disk plays an important role) 
 and Lu-Pan \cite{LuPa}. \\
Notice that the Dirichlet problem has been analyzed intensively in the last years with  similar techniques (see \cite{SoSo,BW24} and references therein) but let us emphasize that the phenomena observed in the  Neumann case are quite different \cite{CLPS2,FH3,FPS0}.

 For the analysis of the spectrum of $H_\beta$, using the radial symmetry of the domain $\mathbb D$, we  pass to the polar coordinates
 \[(x_1,x_2)=(r\cos \theta,r\sin \theta)\]
  and  use a decomposition in Fourier series with respect to the variable $\theta$. We are led to consider
  the family of $1D$-operators indexed by $n\in \mathbb Z$
  \begin{equation}\label{eq:Hn}
  H_{n,\beta} := -\frac {d^2}{dr^2} -\frac 1r \frac{d}{dr} +\left(\frac \beta 2  r-\frac{n}{r} \right)^2\,,
  \end{equation}
  with Dirichlet condition at $r=0$ and Neumann condition at $r=1$\footnote{For $n=0$, we assume Neumann at $r=0$.}. The operators $ H_{n,\beta}$ are self-adjoint in the Hilbert space $L^2((0,1),r\,dr)$.
 
 For each of these operators, we look at the lowest eigenvalue $\lambda(n,\beta)$ as a function of $\beta$. It follows from standard perturbation theory of self-adjoint operators that each map $\beta\mapsto \lambda(n,\beta)$ is real-analytic. We then get the lowest eigenvalue from 
  \begin{equation}\label{eq:infimum}
  \lambda (\beta) =\inf_{n\in\mathbb Z} \lambda(n,\beta)\,.
  \end{equation}
As explained in Section \ref{s:curves}, we can restrict ourselves to $\beta\ge0$ and $n\in\mathbb N$ without loss of generality.  
 
We often find it more convenient to consider auxiliary maps, defined in $(0,+\infty)$ by
\[\beta\mapsto\eta(n,\beta):=\frac{\lambda(n,\beta)}{\beta}\]
and
\[\beta\mapsto\eta(\beta):=\frac{\lambda(\beta)}{\beta}.\]

Our first goal (suggested by previous work and by numerical computations) is to understand the intersection of the curves $\beta\mapsto \lambda(n,\beta)$ and $\beta\mapsto \lambda(n+1,\beta)$ (or, equivalently, of the curves $\beta\mapsto \eta(n,\beta)$ and $\beta\mapsto \eta(n+1,\beta)$). We give a complete proof of a formula due to the physicist D. Saint-James \cite{S-J}.

\newpage

\begin{theorem} \label{theorem:sj}
	Let $n\in \mathbb N$ and let us assume that $\beta^*$ is a solution in $[0,+\infty)$ of the equation 
	\[\lambda(n,\beta)=\lambda(n+1,\beta).\]
	Then, with the notation
	\[\eta^*:=\eta(n,\beta^*)=\eta(n+1,\beta^*),\]
	we have
	\[\beta^*=2\eta^*+2n+1+\sqrt{(2\eta^*+1)^2+8n\eta^*},\]
and, as a consequence, $\beta^*>2(n+1)$. 
\end{theorem}

 We then try to understand which $n$ realizes the infimum \eqref{eq:infimum} according to the value of $\beta$. We  combine the results of M. Dauge and B. Helffer \cite{DH1,DH2}, on eigenvalues variation for Sturm-Liouville operators, with Theorem~\ref{theorem:sj}, to obtain the following result.
 
\begin{theorem} \label{theorem:groundstate} 
 For any $n\in \mathbb N$, there exists a unique solution $\beta_n $ in $[0,+\infty)$ of the equation 
	\[\lambda(n,\beta)=\lambda(n+1,\beta).\]

In addition, the sequence $(\beta_n )_{n\in \mathbb N}$ is strictly increasing, and 
\begin{enumerate}[(i)]
\item for all $\beta\in[0,\beta_0]$, $\lambda(\beta)=\lambda(0,\beta)$;
	\item for all $n\in\mathbb N\setminus\{0\}$, for all $\beta\in[\beta_{n-1},\beta_n ]$, $\lambda(\beta)=\lambda(n,\beta)$.
\end{enumerate}

Furthermore, the equation $\lambda(k,\beta)=\lambda(\beta)$, with $\beta\ge0$, is satisfied only for $k=n$ when $\beta\in(\beta_{n-1},\beta_n )$, with $n>0$, (only for $k=0$ when $\beta\in[0,\beta_0)$) and only for $k\in\{n,n+1\}$ when $\beta=\beta_n $ . 
\end{theorem}

   Finally, we study the following conjectures.
   
     \begin{conjecture}\label{conjUpperBound}
	For all $\beta>0$,
   $$ \eta(\beta) < \Theta_0\,.$$
   \end{conjecture}
   
   In the above statement, $\Theta_0$ is a well-known universal constant, called  the \emph{De Gennes constant} since it was introduced by him  \cite{DG,SdG}  in the sixties in the context of Surface Superconductivity (see for mathematical studies \cite{BH}, and later the book \cite{FH4} and references therein).  This constant is defined 
   \begin{equation}\label{eq:C1aa}
   \Theta_0:= \inf_{\xi\in \mathbb R} \lambda^{DG}(\xi)\,,
   \end{equation}
   where $\lambda^{DG} (\xi)$ is the lowest eigenvalue of the Neumann realization of the harmonic oscillator 
   \begin{equation}\label{eq:defh0}
   \mathfrak h_0(\xi):= D_t^2 + (t+\xi)^2
  \mbox{ on } \mathbb R^+\,, 
   \end{equation} 
   where $D_t:=-i\frac{d}{dt}$.
   Numerically\footnote{See \cite{MPS}.}, $ \Theta_ 0\simeq  0.590106$ and its first role is to give the asymptotics
   \begin{equation}\label{eq:C1ab}
   \lambda(\beta) = \Theta_0 \beta + o(\beta)
   \end{equation}
   as  $\beta$ tends to $+\infty$,  so that $\eta(\beta)\to\Theta_0$.\\
  This property  has been known since the nineties and we actually have more accurate asymptotics for the disk, playing an important role in the analysis of more general domains. Our main interest is in the inequality, which is equivalent to $\lambda(\beta)<\Theta_0\,\beta$ and is an improvement of  the result $\lambda(\beta)<\beta$, proved in \cite[Theorem 2.1]{CLPS1} for a general domain in $\mathbb R^2$. 
   
   We show in  Section \ref{s:interlacing} that Conjecture \ref{conjUpperBound}  can be deduced from the following other conjecture.
   \begin{conjecture}\label{conj:sequence}
 	The sequence $(\eta_n^*)_{n\in\mathbb N}$ is strictly increasing, where
    \[\eta_n^*:=\eta(n,\beta_n )=\eta(n+1,\beta_n )\] 
    and $\beta_n $ is defined in Theorem \ref{theorem:groundstate}.  
   \end{conjecture}

   We also investigate the so-called \emph{strong  diamagnetism} in $\mathbb D$, which states as a conjecture:
   \begin{conjecture}\label{conj:monotone} The map $\beta\mapsto \lambda(\beta)$ is monotonically increasing in the interval $[0,+\infty)$.
   \end{conjecture}
  It follows from the result of S. Fournais and B. Helffer for a general domain in $\mathbb R^2$ \cite{FH2} that the map is ultimately increasing, that is, there exists some $\overline{\beta}>0$ such that $\beta\mapsto \lambda(\beta)$ is increasing in $[\overline{\beta},+\infty)$. The point is to show that, in the case of the disk, the map is increasing in the whole interval $[0,+\infty)$.  Let us note that this property is known to hold when the Neumann boundary condition is replaced with Dirichlet (see the PhD thesis of S. Son \cite[Theorem 3.3.4]{SoSo}). 
  
In  Sections \ref{s8} and \ref{s6},   we present asymptotic and numerical computations supporting these conjectures.  We also show how the combination of  the results in \cite{FH2} with the approach developed in the first sections leads to new relations and permits to control
 the validity of the conjectures in the large magnetic field limit.\\
  
\noindent {\bf  Acknowledgments.\\}
 This problem has been discussed along the years with many colleagues sharing their ideas.  Many thanks to B. Colbois, S. Fournais, A. Kachmar, G. Miranda, M. Persson-Sundqvist, L. Provenzano, and  A. Savo for their contributions. 
 
C. L\'ena received financial support from the INdAM GNAMPA Project  ‘‘\emph{Operatori differenziali e integrali in geometria spettrale}" (CUP E53C22001930001).

The authors are grateful to the referee for his or her careful reading of the manuscript and for suggesting numerous improvements.
 \\

\section{Eigenvalue curves}\label{s:curves}
 
We have, for the operator $ H_{n,\beta}$ defined by Equation \eqref{eq:Hn}, 
\[ H_{n,\beta}=H_{-n,-\beta},\]
so that, from Equation \eqref{eq:infimum}, we get   $$\lambda(-\beta)=\lambda(\beta)\,.$$ In addition, a direct computation shows that 
\[H_{-n,\beta}=H_{n,\beta} +2n\beta,\]
and therefore, in particular
\[\lambda(-n,\beta)=\lambda(n,\beta)+2n\beta.\]
For $\beta>0$, this implies that the infimum
\[\lambda(\beta)=\inf_{n\in\mathbb Z}\lambda(n,\beta)\]
cannot be realized by a negative integer. Thus, if we want to study the map $\beta\mapsto\lambda(\beta)$, it is enough to study the maps $\beta\mapsto \lambda(n,\beta)$ for $n\ge0$ and $\beta\ge0$.

For $\beta=0$, we recover the usual Laplacian, with Neumann boundary conditions.  We therefore assume in the rest of this section that $\beta>0$. The Sturm-Liouville problem corresponding to $ H_{n,\beta}$ is
\begin{equation}\label{eqSL}
\left\{
\begin{array}{rl}
	-f''-\frac1rf'+\left(\frac{n}{r}-\frac{\beta}2r\right)^2f=&\lambda f \mbox{\ \ in \ \ }(0,1)\,,\\
	f(0)=0\mbox{ if }n>0,\ &f'(0)=0\mbox{ if }n=0\,,\\
	f'(1)=&0\,.
\end{array}
\right.
\end{equation}
We write $r\mapsto f_{n,\beta}(r)$ for the positive and normalized eigenfunction of \eqref{eqSL} associated with $\lambda(n,\beta)$.

We can express it using the Kummer function $z\mapsto M(a,b,z)$ (with the conventions and notation of the  \emph{Digital Library of Mathematical Functions} \cite[Chapter 13]{DLMF}):
\begin{equation}\label{eqFKummer}
	f_{n,\beta}(r)=C_n^\beta r^n e^{-\frac{\beta}4r^2}M\left(\frac12-\frac{\lambda}{2\beta},n+1,\frac{\beta}2r^2\right)
\end{equation}
where $\lambda=\lambda(n,\beta)$. \\
 Notice (cf \cite{MOS1966}) that when $a<c$,
 \begin{equation}\label{eq:kummer}
M (a,c,z)= \frac{\Gamma(c)}{\Gamma(c-a)\Gamma(a)} \int_0^1 e^{zt} t^{a-1} (1-t)^{c-a-1} dt\,.
\end{equation}

 Alternatively, we can use the Whittaker function $z\mapsto M_{\kappa,\mu}(z)$ (see for instance \cite[Equation (A.5)]{KLPS}\footnote{The authors thank Mikael Persson Sundqvist for communicating to them unpublished notes and programs.}):
\begin{equation}\label{eqFWhittaker}
	f_{n,\beta}(r)=\frac{D_n^\beta}r M_{\frac{\lambda}{2\beta}+\frac{n}2,\frac{n}{2}}\left(\frac{\beta}{2}r^2\right). 
\end{equation}
The constants $C_n^\beta$ and $D_n^\beta$, depending on $n$ and $\beta$, are chosen so that 
\begin{equation*}
	\int_0^1 f_{n,\beta}(r)^2\,r\,dr=1.
\end{equation*}

The eigenvalues of \eqref{eqSL} (and in particular the first eigenvalue, $\lambda(n,\beta)$) are the roots of the equation
\begin{equation}\label{eqEV}
	(f_{n,\beta})'(1)=0
\end{equation}
(where the left-hand side is understood as the function of $\lambda$ defined by Equation \eqref{eqFWhittaker}).
As an illustration, we present in Figure \ref{figEV}  the curves 
$$\beta\mapsto \eta(n,\beta)=\frac{\lambda(n,\beta)}{\beta}$$ for $n=0,1,\dots,20\,.$
They where drawn by cubic spline interpolation of a finite number of sample points, obtained from solving Equation \eqref{eqEV} numerically (with computations  done in \emph{Mathematica}\footnote{\emph{Wolfram 14.0 For Desktop}. Version: 14.0.0.0.}).

\begin{figure}
\begin{center}
	\includegraphics[width=.85\textwidth]{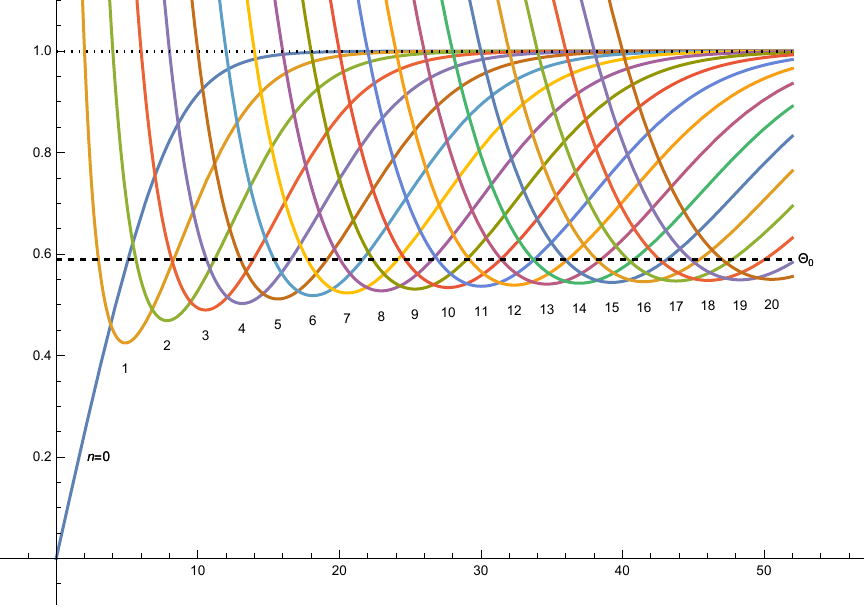}
	\caption{Plot of $\beta\mapsto\eta(n,\beta)=\frac{\lambda(n,\beta)}{\beta}$ for $n$ from 0 to 20. The black dotted line and the black dashed line are the constants $1$ and $\Theta_0$ respectively.}
	\label{figEV}
\end{center}
\end{figure}

In order to study the functions $\beta\mapsto \eta(n,\beta)$, let us give an alternative interpretation of $\eta(n,\beta)$. We perform the change of variable $\rho=\beta^{1/2} r$ and divide by $\beta$. We obtain a new Sturm-Liouville operator $G_{n,\beta}$, which is the Dirichlet-Neumann (Neumann-Neumann if $n=0$) realization in $(0, \beta^{1/2})$ of the differential operator
$$
-\frac{d^2}{d\rho^2} - \frac 1 \rho \frac{d}{d\rho} + \Big( \frac n \rho -\frac \rho 2 \Big)^2\,.
$$
This new operator is self-adjoint in the Hilbert space $L^2((0, \beta^{1/2}),\rho\, d \rho)$ (we recall that the measure has a weight) and $\eta(n,\beta)$ is its first eigenvalue.

\section{Properties of the eigenvalue curves}\label{s3}

\subsection{Asymptotic behavior}

\begin{proposition} \label{propSmall} Under the assumption $n>0$,  we have, as $\beta\to0^+$,
\[ \eta(n,\beta)=\frac{(j_{n,1}')^2}{\beta}(1+o(1)),\]
where $j_{n,1}'$ is the smallest positive zero of $J_n'$, the derivative of the Bessel function of the first kind $J_n$.

In addition, $\eta(0,\beta)\to 0$ as $\beta\to0^+$.
\end{proposition}

\begin{proof} By definition,
$\eta(n,\beta)=\frac{\lambda(n,\beta)}{\beta}$.
From standard perturbation theory, the eigenvalues of $ H_{n,\beta}$ converge to the eigenvalues of $H_{n,0}$ as $\beta\to0$. In particular, $\lambda(n,\beta)\to\lambda(n,0)$. It is easily checked that $\lambda(n,0)=(j_{n,1}')^2$ (see \cite{H-PS} and references therein).

The last limit follows from the well-known bound $\lambda(0,\beta)\le \frac{1}{8}\beta^2$.  This can be obtained by using the constant function $f(r)=1$ in the associated Rayleigh quotient:
\[\lambda(0,\beta)\le\frac{\int_0^11\,( { H_{0,\beta}1})\,r\,dr}{\int_0^11^2\,r\,dr}=\frac{\beta^2}{4}\,\frac{\int_0^1r^3\,dr}{\int_0^1r\,dr}=\frac{\beta^2}{8}.\qedhere\]
\end{proof}

The following result is a special case of a theorem proved by A. Kachmar and G. Miranda \cite[Theorem 1]{Mi} with techniques inspired by \cite{BH}.

\begin{proposition} \label{propLarge} As $\beta\to+\infty$,
\[\lambda(n,\beta)=\beta- e^{-\frac{\beta}{2}}\left(\frac{\beta^{n+2}}{2^nn!}+o\left(\beta^{n+2}\right)\right).\]
\end{proposition}

From this proposition we have, as $\beta\to+\infty$,
\begin{equation}\label{eq:3.7} \eta(n,\beta)=1-  e^{-\frac{\beta}{2}} \left(\frac{\beta^{n+1}}{2^nn!}+o\left(\beta^{n+1}\right)\right).
\end{equation}
 Notice that the Dirichlet case is considered in \cite{BW24}.

\subsection{Variations}

In Section \ref{s:curves}, we have defined the operator $G_{n,\beta}$ and the map $\beta\mapsto \eta(n,\beta)$ for $\beta\in (0,+\infty)$ and $n\in \mathbb N$. However, these definitions make sense for any real $n\in[0,+\infty)$. We will use the extended definitions at some points of our analysis.  In particular, it follows from standard perturbation theory for self-adjoint operators that the map $(n,\beta)\mapsto \eta(n,\beta)$ is smooth, even real-analytic, and the Feynman-Hellmann formula gives us the derivative of $n\mapsto\eta(n,\beta)$, for a given $\beta>0$:
\begin{equation}\label{eqDern}
	\partial_{n} \eta(n,\beta)=\int_0^{\beta^{1/2}}\left(\frac{2n}{\rho^2}-1\right)g_{n,\beta}(\rho)^2\rho\,d\rho\,,
\end{equation}
where $g_{n,\beta}$ is the positive and normalized eigenfunction of the operator $G_{n,\beta}$ associated with $\eta(n,\beta)$.

The work of Dauge-Helffer \cite{DH1,DH2} gives a rather complete picture of the variation of $\beta\mapsto\eta(n,\beta)$ in $(0,+\infty)$, for a given $n\in\mathbb N$ (the results actually hold without assuming that $n$ is an integer).  The derivatives in the following statements are taken with respect to $\beta$.

\begin{proposition} \label{prop:variation} Let $n\in\mathbb N$.
\begin{enumerate}[(i)]
	\item For all $\beta>0$, 
	\[{\eta'(n,\beta)}=\frac12\, \big(g_{n,\beta}(\beta^{1/2})\big)^2\left(\left(\frac{n}{\beta^{1/2}}-\frac{\beta^{1/2}}{2}\right)^2-\eta(n,\beta)\right).\]
	\item The map $\beta \mapsto \eta(0,\beta)$ is increasing in $[0,+\infty)$.
	\item If $n>0$, the map $\beta\mapsto \eta(n,\beta)$ has a unique  point of minimum, which is denoted by $\beta_{\min}(n)$. In addition $\beta_{\min}(n)$ is non-degenerate and satisfies $$\beta_{\min}(n)> 2n\,.
	$$
 \end{enumerate}
\end{proposition}

To unify the notation, we set 
\begin{equation}\label{eq:n0} 
\beta_{\min}(0):=0\,.
\end{equation}

\begin{proof} Setting $a:=\beta^{1/2}$, $\eta(n,\beta)$ is the first eigenvalue of a Sturm-Liouville operator of the form
\[-\frac1{\rho}\,\frac{d}{d\rho}\left(\rho\,\frac{d}{d\rho}\right)+q_n(\rho),\]
in the weighted $L^2$-space $L^2((0,a),\rho\,d\rho)$, with Neumann boundary condition at $\rho=a$ and either Dirichlet (when $n>0$) or Neumann (when $n=0$) boundary condition at $\rho=0$. We write $\mu_n(a)=\eta(n,a^2)$. As in the reference \cite{DH1}, the right endpoint $a$ is variable while the potential 
\[q_n(\rho)=\left(\frac{n}{\rho}-\frac{\rho}{2}\right)^2\]
is independent of $a$.

According to a variant of the Feynman-Hellmann  formula \cite[Equation ($\rm{F}_\omega$) p. 248]{DH1}, for all $a>0$,
\[\mu'_n(a)=a \, g(a)^2 \left(q_n(a)-\mu_n(a)\right),\]
where $g=g_{n,a^2}=g_{n,\beta}$. Substituting for $\mu_n'(\beta^{1/2})$ in
\[{\eta'(n,\beta)}=\frac{d}{d\beta}\left(\mu_n(\beta^{1/2})\right)=\frac{1}{2\beta^{1/2}}\mu_n'(\beta^{1/2}),\]
we obtain Point (i). Points (ii) and (iii) are direct applications of \cite[Theorem 4.2]{DH1} and \cite[Theorem 4.3]{DH1}, respectively.
\end{proof}
As observed in  Fournais--Persson-Sundqvist  \cite{FPS}, 
  a  variant of the first Feynman-Hellmann  formula \eqref{eqDern} gives
  \[\lambda'(n,\beta)=\int_0^1\left(\frac{\beta}{2}r^2-n\right)f_{n,\beta}(r)^2\,r\,dr\]
  and therefore:
\begin{lemma}\label{lem:fps1}
If $\beta < 2n$, then $\lambda'(n,\beta) <0$.
\end{lemma}
It is also noted in \cite{FPS} that, using 
 the trial function $f(r)=r^n\exp(-\beta r^2/4)$ in the Rayleigh quotient, we get:
\begin{lemma}\label{lem:fps2}
If $\beta > 2n$, $\lambda(n,\beta) < \beta$\,.
\end{lemma}

\subsection{Intersections}

 Let us prove a useful intermediate result towards the proof of  Theorems \ref{theorem:sj} and \ref{theorem:groundstate}.

\begin{proposition} \label{prop:inter} Let $n\in\mathbb N$. There exists at least one positive $\beta^*$ such that $\eta(n,\beta^*)=\eta(n+1,\beta^*)$. In addition, any $\beta$ such that $\eta(n,\beta)=\eta(n+1,\beta)$ necessarily satisfies $\beta>2n$.
\end{proposition}

\begin{proof} Propositions \ref{propSmall} and \ref{propLarge} imply, respectively, that $\eta(n,\beta)<\eta(n+1,\beta)$ for $\beta>0$ small enough and $\eta(n,\beta)>\eta(n+1,\beta)$ for $\beta$ large enough. By continuity, there exists at least one $\beta^*>0$ such that $\eta(n,\beta^*)=\eta(n+1,\beta^*)$.

We prove the second part of the statement by contraposition. Let us use the extended definition of $(m,\beta)\mapsto\eta(m,\beta)$ where $m$ is not necessarily an integer. If $\beta>0$ and $m>0$ are such that $\beta\le 2m$, Formula \eqref{eqDern} shows that $\partial_{m}\eta(m,\beta)>0$. In particular, if $0<\beta\le 2n$, the map $m\mapsto \eta(m,\beta)$ is strictly increasing in $[n,n+1]$, and therefore $ \eta(n,\beta)<\eta(n+1,\beta)$.  
\end{proof}

It is not clear at this stage that there is no more than one point of intersection between the two curves.  We show this in the next two sections.

\subsection{Saint-James Formula}\label{s2}

\begin{figure}[htb]
  \includegraphics[width=\textwidth]{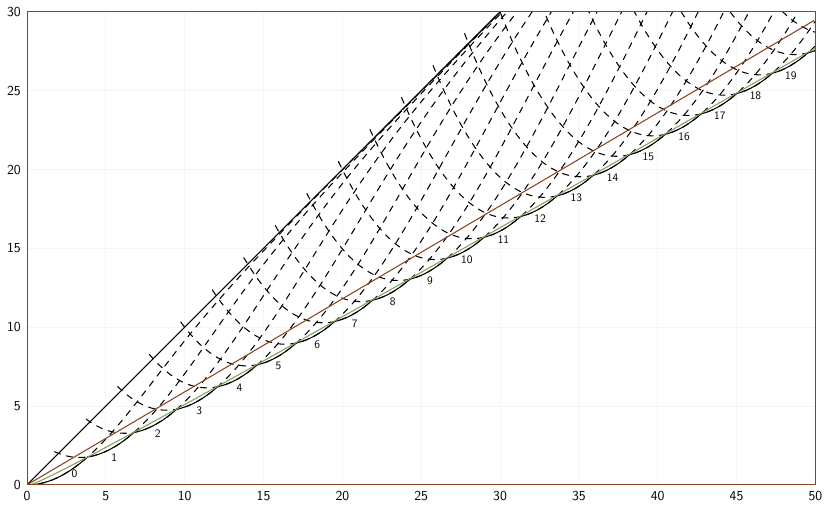}
  \caption{Saint-James picture recovered by M. Persson-Sundqvist}
  \label{fig:sj}
 \end{figure}

The goal of this section is to prove Theorem \ref{theorem:sj} (illustrated by Figure \ref{fig:sj}). We use a representation for the first eigenfunction of $G_{n,\beta}$ deduced from \eqref{eqFKummer} by the relation $$g_{n,\beta}(\rho)=\beta^{-1/2}f_{n,\beta}(\beta^{-1/2}\rho)\,,$$ that is
 \begin{equation}
\label{eq:gn}g_{n,\beta}(\rho)=K_n^\beta\rho^n e^{-\frac{\rho^2}4}M\left(\frac12-\frac{\eta}{2},n+1,\frac{ \rho^2}{2}\right)
\end{equation}
(with $\eta=\eta(n,\beta)$ and $K_n^\beta=\beta^{-(n+1)/2}C_n^\beta$, where $C_n^\beta$ is the normalization constant in Equation \eqref{eqFKummer}). \\
We obtain
\begin{equation*}
	 (g_{n,\beta})'(\rho)=\left(\frac{n}{\rho}-\frac{\rho}{2}\right)g_{n,\beta}(\rho)+ K_n^\beta\,\rho^{n+1} e^{-\frac{\rho^2}{4}}\partial_z M\left(\frac12-\frac{\eta}{2},n+1,\frac{\rho^2}{2}\right).	
\end{equation*}
Recalling the general formula for the derivative of Kummer functions \cite[\href{https://dlmf.nist.gov/13.3.E15}{(13.3.4)}]{DLMF}:
\begin{equation*}
\partial_z M(a,b,z)=\frac{a}{b}M(a+1,b+1,z),
\end{equation*}
we finally get
\begin{equation*}
(g_{n,\beta})'(\rho)=\left(\frac{n}{\rho}-\frac{ \rho}{2}\right)g_{n,\beta}(\rho)+ K_n^\beta\,\frac{\frac{1}{2}-\frac{\eta}{2}}{n+1} \rho^{n+1} e^{-\frac{\beta \rho^2}{4}} M\left(\frac32-\frac{\eta}{2},n+2,\frac{ \rho^2}{2}\right)
\end{equation*}
The Neumann condition at $\beta^{1/2}$ (i.e. $(g_{n,\beta})'(\beta^{1/2})=0$) leads, after simplifications, to the implicit equation for the eigenvalues $\eta$:
\begin{equation*}	
	(n+1)(n-x)M(\nu,n+1,x)+2x\nu M(\nu+1,n+2,x)=0,
\end{equation*}
with the notation
\begin{align*}
	\nu&:=\frac{1}{2}(1-\eta),\\
	x&:=\frac{\beta}{2}.
\end{align*}

 Let $(\beta,\eta)$ be a point of intersection of the analytic curves $\beta\mapsto\eta(n,\beta)$ and $\beta\mapsto\eta(n+1,\beta)$.
We first note that, according to Proposition \ref{prop:inter}, we have $\beta>2n$. As a consequence, according to Lemma~\ref{lem:fps2}, $\lambda(n,\beta)<\beta$, so that $0<\eta<1$.  The corresponding $(x,\nu)$ satisfies the system of equations
\begin{align}
	\label{eqSys1}(n+1)(n-x)M(\nu,n+1,x)+2x\nu M(\nu+1,n+2,x)&=0\, ,\\
	\label{eqSys2}(n+2)(n+1-x)M(\nu,n+2,x)+2x\nu M(\nu+1,n+3,x)&=0\,
\end{align}
 (we have $\nu=\frac12(1-\eta)\in(0,\frac12)$, so the Kummer functions are well-defined).
\begin{remark} \label{rem:nonzero}  Since $(g_{n,\beta})'(\beta^{1/2})=0$ and since $g_{n,\beta}$ is a non-zero solution of a second-order differential equation, we necessarily have $g_{n,\beta}(\beta^{1/2})\neq0$, or equivalently $M(\nu,n+1,x)\neq0$.
\end{remark}

We proceed in the following way to obtain an algebraic equation satisfied by $x$ and $\nu$.
\begin{itemize}
	\item We use the known recursion relations for the functions $M(a,b,z)$ (see for instance \cite[Section 13.3]{DLMF}) to express the left-hand side in the system as a linear combination of 
	\begin{align*}	
		X&:=M(\nu,n+1,x),\\
		Y&:=M(\nu,n+2,x).
	\end{align*}
	\item We write that the determinant of the resulting system is zero (since $(X,Y)\neq(0,0)$, according to Remark \ref{rem:nonzero})  and we obtain the desired equation.
\end{itemize}
From the recursion relation \cite[\href{https://dlmf.nist.gov/13.3.E4}{(13.3.4)}]{DLMF},
\begin{equation}\label{eqRec1}
	xM(\nu+1,n+3,x)=(n+2)M(\nu+1,n+2,x)-(n+2)M(\nu,n+2,x)
\end{equation}
and from \cite[\href{https://dlmf.nist.gov/13.3.E3}{(13.3.3)}]{DLMF}
\begin{equation}\label{eqRec2}
	\nu M(\nu+1,n+2,x)=(n+1)M(\nu,n+1,x)+(\nu-n-1)M(\nu,n+2,x).
\end{equation}
Inserting Equation \eqref{eqRec2} in Equation \eqref{eqSys1}, we get
\begin{equation}\label{eqXY1}
	(n+1)(n+x)X+2x(\nu-n-1)Y=0.
\end{equation}
Inserting Equation \eqref{eqRec1} in Equation \eqref{eqSys2}, we get
\begin{equation*}
	2\nu M(\nu+1,n+2,x)+(n+1-2\nu-x)Y=0.
\end{equation*}
Inserting Equation \eqref{eqRec2} in the previous equation, we obtain
\begin{equation}\label{eqXY2}
	2(n+1)X-(n+1+x)Y=0.
\end{equation}
The system formed by Equations \eqref{eqXY1} and \eqref{eqXY2} is singular if and only if
\begin{equation*}
\left|
\begin{array}{cc}
	(n+1)(n+x)&2x(\nu-n-1)\\
	2(n+1)&-(n+1+x)
\end{array}
\right|=0.
\end{equation*}
This gives the equation 
\begin{equation}\label{eqSingularNu}
	x^2+(4\nu-2n-3)x+n(n+1)=0,
\end{equation}
which can be rewritten
\begin{equation}\label{eqSingularEta}
	x^2-(2\eta+2n+1)x+n(n+1)=0.
\end{equation}
We revert to the variable $\beta$ and we obtain 
\begin{equation}\label{eq:sj1}
	\beta^2-2(2\eta+2n+1)\beta+4n(n+1)=0 \,.
\end{equation}
A similar proof  of the same formula is given in \cite{KLPS}, using Whittaker functions rather than Kummer functions.

Solving for $\beta$  and using again the necessary condition $\beta>2n$, we finally  obtain Saint-James Formula\footnote{Saint-James  \cite{S-J} gives this formula in a footnote, writing $h$ for $\beta$ and $\lambda$ for $\eta$, with the following justification ``En utilisant les propri\'et\'es des fonctions hyperg\'eom\'etriques confluentes". He probably had to prove first \eqref{eq:sj1} but this equation does not appear explicitly in \cite{S-J}.}:
\begin{equation}\label{eq:sj1aa}
\beta=	 2\eta+2n+1 + \sqrt{ (2\eta+1)^2 + 8 \eta n}\,.
\end{equation}

In addition, since $\eta> 0$, Equation \eqref{eq:sj1aa} immediately implies \[\beta>2(n+1).\] This completes the proof of Theorem \ref{theorem:sj}.  

\begin{remark}
We can rewrite \eqref{eq:sj1} in the form
$$
(\beta- (2n+1))^2= 4 \lambda +1\,,
$$
which holds at the crossing point of $\beta\mapsto\eta(n,\beta)$ and $\beta\mapsto\eta(n+1,\beta)$, with $\lambda=\beta \eta$.\\
We then get an alternative form of Saint-James formula: 
\begin{equation}\label{eq:sj1a}
\beta= (2n+1) + \sqrt{4\,\lambda +1}\,.
\end{equation}
\end{remark}

\subsection{Interlacing between intersections and minima}
\label{s:interlacing}
We study in this section  the relation between the points of intersection and the minima $\beta_{\min}(n)$. Our results give in particular a proof of Theorem~\ref{theorem:groundstate}.  The derivatives are again taken with respect to $\beta$.

\begin{proposition} \label{prop:interlacing} Let $n\in\mathbb N$ and let $\beta^*$ be a solution in $(0,\infty)$ of the equation
\[\eta(n,\beta)=\eta(n+1,\beta).\]
Then,
\[\eta'(n,\beta^{*})>0\]
and
\[\eta'(n+1,\beta^{*})<0.\]
\end{proposition}

\begin{proof} We use the notation $\eta^*_n=\eta(n,\beta^*)=\eta(n+1,\beta^*)$ and $x^*:=\beta^*/2$ to make the formulas more readable. A direct computation, using Point (i) of Proposition \ref{prop:variation}, gives 
\[{ \eta'(n,\beta^*)}=c_n\frac{(x^*)^2-(2\eta^*_n+2n)x^*+n^2}{2x^*}\]
and
\[{ \eta'(n+1,\beta^*)}=c_{n+1}\frac{(x^*)^2-(2\eta^*_n+2n+2)x^*+(n+1)^2}{2x^*},\]
with $c_n$ and $c_{n+1}$ positive factors. Taking into account that $n$, $x^*$ and $\eta^*$ satisfy Equation \eqref{eqSingularEta}, we obtain
\[{ \eta'(n,\beta^*)}=c_n\frac{x^*-n}{2x^*}\]
and
\[{ \eta'(n+1,\beta^*)}=c_{n+1}\frac{n+1-x^*}{2x^*}.\]
According to Theorem  \ref{theorem:sj}, $x^*>n+1$, and therefore  $\eta'(n,\beta^*)>0$ and $\eta'(n+1,\beta^*)<0$.
\end{proof}

The following theorem is now easy to prove, and implies Theorem \ref{theorem:groundstate}. We use the convention $\beta_{\min}(0)=0$.

\begin{theorem} \label{theorem:interlacing}
For any $n\in \mathbb N$, there exists a unique solution $\beta_n $ in $[0,+\infty)$ of the equation 
	\[\lambda(n,\beta)=\lambda(n+1,\beta).\]
In addition,
\begin{enumerate}[(i)] 
\item the sequences $(\beta_{\min}(n))$ and $(\beta_n )$ are both increasing,
\item for any $n\in\mathbb N$, $\beta_{\min}(n)<\beta_n <\beta_{\min}(n+1)$,
\item $\lambda(\beta)=\lambda(0,\beta)$ when $\beta\in[0, \beta_0]$ and, for $n\in\mathbb N\setminus\{0\}$, $\lambda(\beta)=\lambda(n,\beta)$ when $\beta\in[\beta_{n-1},\beta_n ]$. Furthermore, if $\beta\notin\{\beta_n \,:\,n\in \mathbb N\}$, there is a unique $k\in\mathbb N$ such that $\lambda(\beta)=\lambda(k,\beta)$. 
\end{enumerate}

\end{theorem}

\begin{proof}
Let $\beta^*$ be a solution of
\[\eta(n,\beta)=\eta(n+1,\beta).\]
The signs of the derivatives established in Proposition \ref{prop:interlacing} imply that 
\[\beta_{\min}(n)<\beta^*<\beta_{\min}(n+1)\] 
(the case $n=0$ is included using the convention that $\beta_{\min}(0)=0$). Let us note that, since at least one solution exists according to Proposition \ref{prop:inter}, we have shown $\beta_{\min}(n)<\beta_{\min}(n+1)$.

By the same argument, any other solution must also belong to the interval \[(\beta_{\min}(n),\beta_{\min}(n+1)).\] Since $\beta\mapsto\eta(n,\beta)$ is strictly increasing and $\beta\mapsto \eta(n+1,\beta)$ strictly decreasing in the interval $(\beta_{\min}(n),\beta_{\min}(n+1))$, there cannot be another solution. We now use the notation $\beta_n :=\beta^*$ and we note that we have proved Point (ii). Since the sequence $(\beta_{\min}(n))$ is increasing, so is the sequence $(\beta_n )$, which proves Point (i).

To prove Point (iii), we remark that multiplying the values $\eta(n,\beta)$ by $\beta$ does not change their ordering, so that $\lambda(n,\beta)<\lambda(n+1,\beta)$ for $\beta<\beta_n $ and $\lambda(n,\beta)>\lambda(n+1,\beta)$ for $\beta>\beta_n $. Combining this with the fact that the sequence $(\beta_n )$ is increasing, we first deduce that, for $\beta\in(0, \beta_0)$,
\[\lambda(0,\beta)<\lambda(1,\beta)<\lambda(2,\beta)<\dots\,,\]
and in particular $\lambda(\beta)=\lambda(\beta,0)$. Then, for any $n\in\mathbb N\setminus\{0\}$ and $\beta\in(\beta_{n-1},\beta_n )$, we find
\[\lambda(n,\beta)<\lambda(n+1,\beta)<\lambda(n+2,\beta)<\dots\]
and 
\[\lambda(0,\beta)>\lambda(1,\beta)>\dots>\lambda(n-1,\beta)>\lambda(n,\beta).\]
In particular, $\lambda(\beta)=\lambda(n,\beta)$. This establishes Point (iii) when $\beta$ is not a point of intersection. The full statement follows by continuity.
\end{proof}

 Combining Point (iii) of the previous theorem with the inequality $\beta_{n-1}>2n$ (deduced from Saint-James formula \eqref{eq:sj1aa}) and with Lemma \ref{lem:fps2}, we obtain, for all $\beta>0$, $$\lambda(\beta) < \beta\,.$$ This inequality is proved for a general domain in $\mathbb R^2$ by Colbois-L\'ena-Provenzano-Savo \cite[Theorem 2.1]{CLPS1}.

 Let us now recall the definition 
 $$
  \eta_n^*:= \eta(n,\beta_n )=\eta(n+1,\beta_n )
 $$
 appearing in  the statement of Conjecture \ref{conj:monotone}.
 
 Since $\eta(\beta)=\eta(n+1,\beta)$ for $\beta\in[\beta_n ,\beta_{n+1}]$ and since $\beta\mapsto\eta(n+1,\beta)$ is decreasing in $[\beta_n ,\beta_{\min}(n+1)]$ and increasing in $[\beta_{\min}(n+1),\beta_{n+1}]$, we have
 \begin{equation}\label{eq:bound-eta-n}
 	\eta(n+1,\beta)\le\max\{\eta_n^*,\eta_{n+1}^*\}\mbox{ for all }\beta\in[\beta_n ,\beta_{n+1}]\,.
 \end{equation}
 In addition,
  \begin{equation}\label{eq:bound-eta-0}
 	\eta(\beta)\le\eta_{0}^*\mbox{ for all }\beta\in(0,\beta^*_{0}]\,,
 \end{equation}
 since $\eta(\beta)=\eta(0,\beta)$ and $\beta\mapsto \eta(0,\beta)$ is strictly increasing in $(0,\beta_0^*]$.
 
Conjecture \ref{conjUpperBound} is therefore equivalent to 
\begin{equation}	\label{eq:bound-eta}
	\eta^*_n<\Theta_0\mbox{ for all }n\in \mathbb N.
\end{equation}
From known asymptotic results  \eqref{eq:C1ab}, $\eta_n^*$ tends to $ \Theta_0$. Therefore, it would be enough to prove that the sequence $(\eta^*_n)$ is strictly increasing to obtain inequality \eqref{eq:bound-eta}. This shows that Conjecture \ref{conj:sequence} implies Conjecture \ref{conjUpperBound}.\\

  Notice also that we can deduce from \eqref{eq:sj1a}
 that if Conjecture \ref{conj:monotone}  holds, i.e. if  $\beta\mapsto\lambda(\beta)$ is increasing, then 
 \begin{equation}\label{eq:5.5}
 \beta_{n+1} \geq 2+ \beta_n \,.
 \end{equation}
In particular, this holds for $n$ large enough, using the Fournais-Helffer \cite{FH1} monotonicity  result, which will be recalled in Section \ref{s8}  and analyzed in more detail. This will also be  confirmed numerically, for $n\leq 400$.

\section{Asymptotic results for the disk revisited}\label{s8}
Starting from the asymptotic analysis given in \cite{FH1} we show that these asymptotics can be refined at the intersection points i.e. for the sequence $\beta_n $ as $n\rightarrow +\infty$. We then show how the comparison with the information coming from Saint-James formula leads to complete asymptotics of $\beta_n $ and $\lambda(\beta_n )=\lambda(n,\beta_n )$. This validates asymptotically a conjecture on the monotonicity of $\beta_{n+1}-\beta_n $. 
 \subsection{The case of the disk (reminder after \cite{FH1})}\label{disc}~\\
We recall that it follows of the proof of Theorem 2.5  in \cite{FH1}:
 \begin{theorem}[Eigenvalue asymptotics for the disk]
\label{theorem:disc}~\\
Suppose that $\mathbb D$ is the unit disk and define $\delta(m,\beta)$, for $m\in {\mathbb Z}$, $\beta>0$, by
\begin{align}
\label{eq:delta}
\delta(m,\beta) = m - \tfrac{\beta}{2} - \xi_0 \sqrt{\beta}\,.
\end{align}
Then there exist (computable) constants $C_0, \delta_0 \in {\mathbb R}$ such that if
\begin{align}
\Delta_\beta = \inf_{m \in {\mathbb Z}} | \delta(m,\beta) - \delta_0|\;,
\end{align}
then
\begin{align}\label{eq:asympFoHe}
\lambda(\beta) = \Theta_0 \beta - C_1 \sqrt{\beta} +
3C_1 \Theta_0^{1/2} \big( \Delta_\beta^2 + C_0\big) + {\mathcal O}(\beta^{-\frac{1}{2}})\;.
\end{align}
\end{theorem}
Here, with \eqref{eq:C1aa} in mind, we  recall that
 $\xi_0 < 0$ is the point where $\xi \mapsto \lambda^{DG} (\xi)$ attains its unique (non degenerate) minimum and that  $C_1$ is given by
\begin{align}\label{eq:C1}
C_1 = \frac{u_0^2(0)}{3}\;.
\end{align}
where  $u_0=\phi_{\xi_0}$, is  the positive, normalized eigenfunction associated with $\lambda^{DG} (\xi_0)$.\\
Numerically  (see \cite{Bo1}, \cite{FH1} or \cite{MPS})
$$
C_1 \sim 0.254\,,\, \xi_0 \sim -0.768 \,.
$$
We also recall that
 \begin{equation}\label{eq:C1a}
 \Theta_0 =\xi_0^2\,.
 \end{equation}

Although not used here, notice that the $3D$-case was analyzed by Helffer-Morame \cite{HM04} (see  also \cite[Chap. 9]{FH4}) and, for the version with three terms, in \cite{FP}.

\subsection{Theorem \ref{theorem:disc} revisited}\label{subsec:disc}~\\
We now revisit  the proof of Theorem~\ref{theorem:disc} as given in \cite{FH1} (see their Appendix C and  also \cite{FP}) in order to improve its conclusions  when considering $\beta=\beta_n $.

We  start  with the Neumann operator
 $$\mathfrak h_0(\zeta) = -\frac{d^2}{d\tau^2} + (\tau + \zeta)^2$$ on $L^2({\mathbb R}_{+},d\tau)\;.$ 

Let $D(t) = \{ x \in {\mathbb R}^2 \,| \, |x| \leq t\}$ be the disc with radius $t$.
Let $\widetilde{Q}_\beta$ be the quadratic form
$$
\widetilde{Q}_\beta[u] = \int_{D(1) \setminus D(\frac{1}{2})} \big|(-i\nabla - \beta {\bf A})u\big|^2\,dx\;,
$$
with domain $\{u \in H^1(D(1) \setminus D(\frac{1}{2})) \,| \, u(x) = 0 \text{ on } |x|=\frac{1}{2} \}$.
Let $\tilde{\lambda}(\beta)$ be the lowest eigenvalue of the corresponding self-adjoint operator (Friedrichs extension).
Using the Agmon estimates in the normal direction (see Theorem 4.1 in \cite{FH1}), it can be proved 
\begin{align}
\label{eq:trekant}
\lambda (\beta) = \tilde{\lambda}(\beta) + {\mathcal O}(\beta^{-\infty})\;.
\end{align}

By changing to boundary coordinates (if $(r,\theta)$ are usual polar coordinates, then $t=1-r$, $s=\theta$), the quadratic form $\widetilde{Q}_\beta[u]$ becomes,
\begin{align}
\widetilde{Q}_\beta[u] &= \int_0^{2\pi} ds \int_0^{1/2} dt \,(1-t)^{-1} | (D_s - \beta \tilde{A}_1)u|^2+
(1-t) |D_t u|^2\;, \\
\| u \|_{L^2}^2 &=  \int_0^{2\pi} ds \int_0^{1/2} dt \,(1-t) |u|^2\;,
\quad \tilde{A}_1 = \tfrac{1}{2} - t + \tfrac{t^2}{2}\;.\nonumber
\end{align}
Performing the scaling $\tau = \sqrt{\beta} \, t$ and decomposing in Fourier modes, 
$$
u = \sum_m e^{im s} \phi_m(t)\;,
$$ 
leads to 
\begin{align}
\label{eq:stjerne}
\tilde{\lambda}(\beta) = \beta \inf_{m \in {\mathbb Z}} e_{\delta(m,\beta), \beta}\;.
\end{align}
Here the function $\delta(m, \beta)$ was defined in \eqref{eq:delta} 
and $e_{\delta,\beta}$ is the lowest eigenvalue of the operator $\mathfrak h(\delta,\beta)$ associated with the quadratic form $q_{\delta,\beta}$ on 
$L^2((0, \sqrt{\beta}/2);(1-\beta^{-1/2}\tau)d\tau)$ (with Neumann boundary condition at $0$ and Dirichlet at $\sqrt{\beta}/2$):
\begin{align}
q_{\delta,\beta}[\phi] = \int_0^{\sqrt{\beta}/2} &\Big((1-\tfrac{\tau}{\sqrt{\beta}})^{-1} \big( (\tau+\xi_0) +\beta^{-\frac{1}{2}}(\delta -\tfrac{\tau^2}{2}) \big)^2 |\phi(\tau)|^2 \nonumber\\
&+ (1-\tfrac{\tau}{\sqrt{\beta}}) |\phi'(\tau)|^2\Big)\,d\tau \;.
\end{align}

The self-adjoint Neumann operator ${\mathfrak h}(\delta,\beta)$ associated with 
$q_{\delta,\beta}$ (on the space $L^2((0, \sqrt{\beta}/2);(1-\beta^{-1/2}\tau)d\tau)$) is
\begin{align}
{\mathfrak h}(\delta,\beta)&=
-(1-\tfrac{\tau}{\sqrt{\beta}})^{-1} \frac{d}{d\tau} (1-\tfrac{\tau}{\sqrt{\beta}}) \frac{d}{d\tau}\nonumber\\
&\quad+(1-\tfrac{\tau}{\sqrt{\beta}})^{-2} \big( (\tau+\xi_0) +\beta^{-\frac{1}{2}}(\delta -\tfrac{\tau^2}{2}) \big)^2\;.
\end{align}

We will only consider $\delta$ varying in a fixed bounded set. In this case,  we know  that there exists a $d>0$ such that if $\beta >d^{-1}$, then
the spectrum of $\mathfrak h(\delta,\beta)$ contained in $(-\infty, \Theta_0+d)$ consists of exactly one simple eigenvalue. We will justify this restriction later, using an a priori lower bound given by \cite[p. 193]{FH1}, which we now recall.
\begin{lemma} \label{lemLower} For all $C>0$, there exists positive constants $B_0$ and $D$ such that, if $|\delta|\ge D$ and $\beta\ge B_0$,
\begin{equation*}
 e_{\delta,\beta}\ge \Theta_0-C_1\beta^{-1/2}+C.
\end{equation*}
\end{lemma}

We can formally develop ${\mathfrak h}(\delta,\beta)$ as
$$
{\mathfrak h}(\delta,\beta) = {\mathfrak h}_0 + \beta^{-\frac{1}{2}} {\mathfrak h}_1 + \beta^{-1} {\mathfrak h}_2 + {\mathcal O}(\beta^{-\frac{3}{2}})\;.
$$
with
\begin{align}
\label{eq:hs}
{\mathfrak h}_0 &= -\frac{d^2}{d\tau^2} + (\tau+\xi_0)^2\;, \nonumber\\
{\mathfrak h}_1 &=\frac{d}{d\tau} + 2 (\tau+\xi_0) (\delta -\tfrac{\tau^2}{2}) + 2\tau (\tau+\xi_0)^2\;,\nonumber\\
{\mathfrak h}_2 &= \tau\frac{d}{d\tau} +  (\delta -\tfrac{\tau^2}{2})^2 + 4 \tau (\tau+\xi_0) (\delta -\tfrac{\tau^2}{2})
+ 3 \tau^2 (\tau+\xi_0)^2\;.
\end{align}
Let $u_0$ be the ground state eigenfunction of ${\mathfrak h}_0$ with eigenvalue $\Theta_0$, where ${\mathfrak h}_0$ is considered as a selfadjoint operator on $L^2({\mathbb R}_{+}; d\tau)$ with Neumann boundary condition at $0$.
Let $R_0$ be the regularized resolvent, which is defined by
$$
R_0 \phi = \begin{cases}({\mathfrak h}_0 -\Theta_0)^{-1} \phi\;, & \phi \perp u_0\;, \\ \quad 0\;, &
\phi \parallel u_0\,.
\end{cases}
$$
Here the orthogonality is measured with respect to the usual inner product in $L^2({\mathbb R}_{+}; d\tau)$.

Let $\lambda_1$ and $\lambda_2$ be given by
\begin{align}
\lambda_1 &:= \langle u_0 \,|\,{\mathfrak h}_1 u_0 \rangle\;, &
\lambda_2 &:= \lambda_{2,1} + \lambda_{2,2}\;,\nonumber\\
\lambda_{2,1}&:= \langle u_0 \,|\,{\mathfrak h}_2 u_0 \rangle \;,&
\lambda_{2,2}&:= \langle u_0 \,|\,({\mathfrak h}_1 - \lambda_1) u_1 \rangle \;,
\end{align}
Here the inner products are the usual inner products in $L^2({\mathbb R}_{+}; d\tau)$.
The functions $u_1, u_2$ are given as
\begin{align}
u_1 &= - R_0 ({\mathfrak h}_1 - \lambda_1) u_0\;,
&
u_2 &= - R_0 \big\{ ({\mathfrak h}_1 - \lambda_1) u_1 + ({\mathfrak h}_2 - \lambda_2)u_0 \big\}\;.
\end{align}
Notice that  $u_0 \in {\mathcal S}(\overline{{\mathbb R}_{+}})$ and that $R_0$ maps ${\mathcal S}(\overline{{\mathbb R}_{+}})$ (continuously) into itself.
Therefore, $u_0, u_1, u_2$ (and their derivatives) are rapidly decreasing functions on ${\mathbb R}_{+}$.

Let $\chi \in C_0^{\infty}({\mathbb R})$ be a usual cut-off function, such that
\begin{align}
\chi(t) &= 1 \quad \text{ for } |t|\leq \tfrac{1}{8} \;, &
{\rm supp} \chi &\subset [-\tfrac{1}{4}, \tfrac{1}{4}]\;,
\end{align}
and let $\chi_\beta(\tau) = \chi(\tau \beta^{-\frac{1}{4}})$\,.

Our trial state is defined by
\begin{align}
\label{eq:trial}
\psi := \chi_\beta \big\{ u_0 + \beta^{-\frac{1}{2}} u_1 + \beta^{-1} u_2 \big\}\;.
\end{align}
A calculation (using in particular the exponential decay of the involved functions) gives that
\begin{align}
&\big\| \big\{{\mathfrak h}(\delta,\beta) - \big(\Theta_0 + \lambda_1 \beta^{-\frac{1}{2}} + \lambda_2 \beta^{-1}\big)\big\}
\psi \big\|_{L^2([0, \sqrt{\beta}/2];(1-\beta^{-1/2}\tau)d\tau)} = {\mathcal O}(\beta^{-\frac{3}{2}})\;,\nonumber\\
&\| \psi \|_{L^2([0, \sqrt{\beta}/2];(1-\beta^{-1/2}\tau)d\tau)} = 1 +{\mathcal O}(\beta^{-\frac{1}{2}})\;,
\end{align}
where the constant in ${\mathcal O}$ is uniform for $\delta$ in bounded sets.

Therefore, we have proved that (uniformly for $\delta$ varying in bounded sets)
\begin{align}
e_{\delta, \beta} = \Theta_0 + \lambda_1 \beta^{-\frac{1}{2}} + \lambda_2 \beta^{-1} +{\mathcal O}(\beta^{-\frac{3}{2}})\;.
\end{align}
It remains to calculate $\lambda_1, \lambda_2$ and, in particular, deduce their dependence on $\delta$.

A standard calculation (which can for instance be found in \cite[Section 2]{FH0}) gives that
\begin{align}
\lambda_1=-C_1\;.
\end{align}

It was  harder to calculate $\lambda_2$ explicitly but it is proven in \cite{FH1} that
\begin{equation}\label{eq:bh-1}
\lambda_2(\delta)  = 3 C_1 
\Theta_0^{1/2}\big( (\delta - \delta_0)^2 + C_0\big)\;.
\end{equation}
Remembering \eqref{eq:trekant}, Lemma \ref{lemLower} and \eqref{eq:stjerne}, this finishes the proof of Theorem~\ref{theorem:disc}. A computation of $\delta_0$ will be given later.

The same proof (but  pushing the expansion further as done for instance in \cite{FP}) gives the following extension
\begin{proposition}
For any $2\leq N\in \mathbb N$, uniformly for $\delta$ varying in bounded sets, we have 
\begin{align}\label{eq:4.21}
e_{\delta, \beta} = \Theta_0 + \lambda_1 \beta^{-\frac{1}{2}} + \lambda_2 (\delta) \beta^{-1} +\sum_{j=3}^N \lambda_j(\delta) \beta^{-j/2} + {\mathcal O}(\beta^{-\frac{N+1}{2}})\;.
\end{align}
Moreover the $\lambda_j(\delta)$ are polynomials of $\delta$.
\end{proposition}
\begin{proof}
Following the previous proof (corresponding to $N=2$)
\begin{align}
\label{eq:trialN}
\psi_N := \chi_\beta \big\{ \sum_{j= 0}^N   \beta^{-\frac{j}{2}} u_j \}
\end{align}
and 
\begin{align}\label{eq:trialNz}
&\big\| \big\{{\mathfrak h}(\delta,\beta) - \big(\Theta_0 + \lambda_1 \beta^{-\frac{1}{2}} +\sum_{j=2}^N  \lambda_j(\delta) \beta^{-j/2}\big)\big\}
\psi_N \big\|_{L^2([0, \sqrt{\beta}/2];(1-\beta^{-1/2}\tau)d\tau)} = {\mathcal O}(\beta^{-\frac{N+1}{2}})\;,\nonumber\\
&\| \psi_N \|_{L^2([0, \sqrt{\beta}/2];(1-\beta^{-1/2}\tau)d\tau)} = 1 +{\mathcal O}(\beta^{-\frac{1}{2}})\;,
\end{align}
where the constant in ${\mathcal O}$ is uniform for $\delta$ in bounded sets.\\
Note for later that the normalized eigenfunction $u_{\delta,\beta}$ associated with $e_{\delta,\beta}$ satisfies, for any interval $(0, M)$ and any $N$,
\begin{equation}\label{eq:4.24}
|| u_{\delta,\beta} - u_0 - \sum_{j=1}^N  \beta^{-\frac{j}{2}} u_j||_{H^1((0,M); (1-\beta^{-1/2}\tau)d\tau)}\leq C_{M,N} \beta^{- (N+1)/2}\,,
\end{equation}
for $\beta$ large enough.\\
Here the $u_j$ are in $\mathcal S(\overline{\mathbb R_+})$ and depend polynomially of $\delta$.\end{proof}

To carry our analysis further, we note that the asymptotic results in \cite{FH1} actually provide more information than summarized in Equation \eqref{eq:trekant}. To see this, we first remark that the Fourier decomposition with respect to $\theta$ introduced in Section \ref{s:intro} corresponds to a decomposition of $L^2(\mathbb D)$ into an orthogonal sum of subspaces
\begin{equation*}
	\bigoplus_{m\in\mathbb Z}\mathcal H_m:=\bigoplus_{m\in\mathbb Z}\left\{f(r)e^{im\theta}\left|f\in L^2((0,1),r\,dr)\right.\right\}.
\end{equation*}
Each $\mathcal H_m$ can be seen as an eigenspace of the angular momentum operator 
\begin{equation*}
	-i\left(x_1\pa_{x_2}-x_2\pa_{x_1}\right)=-i\partial _\theta\,,
\end{equation*}
which commutes with the self-adjoint operator $H_\beta$. Therefore, $\mathcal H_m$ (or more precisely its intersection with the domain of $H_\beta$) is an invariant subspace for $H_\beta$, and the operator $H_{m,\beta}$ defined in \eqref{eq:Hn} can be seen as the restriction of $H_\beta$ to $\mathcal H_m$.

Similarly, the Fourier decomposition used to obtain \eqref{eq:stjerne} corresponds to the orthogonal decomposition of $L^2(D(1)\setminus D(\frac12))$ into invariant subspaces 
\begin{equation*}
	\bigoplus_{m\in\mathbb Z}\widetilde{\mathcal H}_m:=\bigoplus_{m\in\mathbb Z}\left\{\phi(t)e^{im s}\left|\phi\in L^2((0,1/2),(1-t)\,dt)\right.\right\}.
\end{equation*}

The uniform radial Agmon estimates used in the proof of \eqref{eq:trekant}  (see again Theorem 4.1 in \cite{FH1}) are rotationally invariant, so the proof in \cite{FH1} applies separately to the invariant subspaces for each $m$. We thus obtain that $e_{\delta(m,\beta),\beta}$ is a close approximation to the eigenvalue $\eta(m\,\beta)$ in the following sense.

\begin{proposition}\label{prop:approx}
Given $\varepsilon\in (0,1)$, there exist positive constants $B_0,\alpha,C$ such that, for each $\beta>B_0$ and $m\in\mathbb Z$ for which the condition $\eta(m,\beta)\le 1-\varepsilon$ is satisfied, we have
\begin{equation*}	
	\left|e_{\delta(m,\beta)}-\eta(m,\beta)\right|\le Ce^{-\alpha \beta^{1/2}}.
\end{equation*}
\end{proposition}

Gathering the preceding results, we get the following.
\begin{corollary}
 For any $N$, we have with $\beta=\beta_n $ and $\delta= \delta(n,\beta_n )$
\begin{equation}\label{eq:bh0}
{\lambda(\beta_n )}=\lambda(n,\beta_n )=\beta\Big( \Theta_0 + \lambda_1 \beta^{-\frac{1}{2}} + \lambda_2(\delta) \beta^{-1} +\sum_{j=3}^N \lambda_j(\delta) \beta^{-j/2} + {\mathcal O}(\beta^{-\frac{N+1}{2}})\Big)\,.
\end{equation}
\end{corollary}
\begin{proof} Since $\eta(n,\beta_n)\to \Theta_0<1$ as $n\to\infty$, the spectral gap condition $\eta(n,\beta)\le 1-\varepsilon$ is satisfied, for some fixed $\varepsilon>0$, when $n$ is large enough. Proposition \ref{prop:approx} therefore tells us that 
\begin{equation}\label{eq:approxBranch}
\eta(n,\beta_n)=e_{\delta(n,\beta_n),n}+\mathcal O\left(\beta_n^{-\infty}\right).
\end{equation}
It then follows from Lemma \ref{lemLower} that
\[\left|\delta(n,\beta_n)\right|\le C\]
for some constant $C$. We can now combine the estimate \eqref{eq:approxBranch} with the asymptotic expansion \eqref{eq:4.21} and multiply by $\beta_n$ to obtain \eqref{eq:bh0}. 
\end{proof}

\begin{remark} Since $\lambda(n,\beta_n)=\lambda(\beta_n)=\lambda(n+1,\beta_n)$, the same reasoning as in the previous proof results in
\begin{equation}\label{eq:asymptBis}
\lambda(n+1,\beta_n )=\beta\Big( \Theta_0 + \lambda_1 \beta^{-\frac{1}{2}} + \lambda_2(\delta) \beta^{-1} +\sum_{j=3}^N \lambda_j(\delta) \beta^{-j/2} + {\mathcal O}(\beta^{-\frac{N+1}{2}})\Big)
\end{equation}
for any $N$, with $\beta=\beta_n $ and $\delta= \delta(n+1,\beta_n )$.
\end{remark}

\begin{theorem}
 There exist  sequences $\hat \delta_j$, $\tilde \delta_j$ ($j\geq 1$) and $\hat \lambda_j$ ($j\geq 3$) such that, for any $N$, we have 
\begin{equation}\label{eq:bh1}
\delta(n,\beta_n) = \delta_0-\frac 12 +\sum_{1\leq j \leq N}\hat \delta_j \, \beta_n ^{-j/2} + {\mathcal O}(\beta_n ^{-\frac{N+1}{2}})\,,
\end{equation}
\begin{equation}\label{eq:bh1a}
\delta(n+1,\beta_n) = \delta_0 + \frac 12 +\sum_{1\leq j \leq N}\tilde \delta_j \, \beta_n ^{-j/2} + {\mathcal O}(\beta_n ^{-\frac{N+1}{2}})\,,
\end{equation}
and 
\begin{equation}\label{eq:bh2}
 \lambda(\beta_n )/\beta_n =  \Theta_0 + \lambda_1 \, \beta_n ^{-\frac{1}{2}} + \lambda_2(\delta_0-\frac 12)\, \beta_n ^{-1} +\sum_{j=3}^N \hat \lambda_j \, \beta_n ^{-j/2} + {\mathcal O}(\beta_n ^{-\frac{N+1}{2}})\,.
\end{equation}
\end{theorem}
\begin{proof}
To get \eqref{eq:bh1}, we simply compare the expansions given for $\lambda(n,\beta_n )$ (i.e $\delta=\delta(n,\beta_n )$ and $\lambda(n+1,\beta_n )$ (i.e. $\delta=\delta(n+1,\beta_n )$).\\

Let us describe  the first step in more detail. Taking the first three terms of the asymptotic expansions \eqref{eq:bh0} and \eqref{eq:asymptBis}, we obtain
\begin{align*}
\lambda(\beta_n ) /\beta_n  &=\Theta_0 - C_1(\beta_n )^{-1/2} + 3 C_1 \Theta_0^{1/2}\big((\delta(n,\beta_n )-\delta_0)^2  + C_0 \big)(\beta_n )^{-1} + \mathcal O(\beta_n^{-3/2})\\
&=\Theta_0 - C_1(\beta_n )^{-1/2} + 3 C_1 \Theta_0^{1/2}\big((\delta(n+1,\beta_n )-\delta_0)^2  + C_0 \big)(\beta_n )^{-1} + \mathcal O(\beta_n^{-3/2})\,.
\end{align*}
From this, we deduce
 \begin{equation}
 \delta (n,\beta_n ) + \delta(n+1,\beta_n ) = 2 \delta_0 + \mathcal O(\beta_n^{-1/2})\,.
 \end{equation}
 which implies using \eqref{eq:delta}, \eqref{eq:bh1}
 and 
 \begin{equation}\label{az1}
 \delta_0 - \delta (n,\beta_n )=\frac 12 +  \mathcal O(\beta_n^{-1/2})\,.
 \end{equation}
 Hence we obtain at the crossing point $\beta_n $
\begin{equation}\label{eq:bh2ab}
\frac{\lambda(\beta_n )}{\beta_n }:= \eta_n^* =\Theta_0 - C_1(\beta_n ) ^{-1/2} + 3 C_1 \Theta_0^{1/2}(\frac 14 + C_0 )(\beta_n )^{-1} +  \mathcal O(\beta_n^{-3/2}) \,.
\end{equation}
All these formulas correspond to an asymptotic as $\beta_n\rightarrow + \infty$. \\
By recursion we get expansions at any order.
\end{proof}

 Notice  that we have not at the moment mixed these results with Saint-James formula. We use it now and in Section \ref{sec:sj-implement}.
 
\begin{remark}
 Combining with   \eqref{eq:sj1a} at $\beta_n $ 
we get, using \eqref{eq:delta}, \eqref{eq:bh1} and \eqref{eq:bh2}, 
\begin{equation}
\delta_0 = \frac 12 \Theta_0^{-1/2} C_1 \sim 0.0975\,.
\end{equation}
This formula seems to be new and was not obtained in \cite{FH1}.
\end{remark}

  \subsection{Implementation of  Saint-James formula} 
  \label{sec:sj-implement}
  In this subsection, we give a complete expansion for $\beta_n $ as an asymptotic series in powers of $n^{-1/2}$.
  \begin{theorem}\label{prop9.1} There is an infinite sequence $\hat \kappa_j$ ($j\geq 0$) such that
\begin{subequations}\label{eq:7.13z}
\begin{equation}
\beta_n  =   2n  + \xi_1  n^\frac 12 + \hat \kappa_0 + \sum_{j\geq 1}\hat \kappa_{j} n^{-j/2} + \mathcal O (n^{-\infty})\,,
\end{equation}
with
\begin{equation}
\xi_1= - 2^{3/2} \xi_0\,,\, \hat \kappa_0= 1- 2 \delta_0 + 2 \xi_0^2\,.
\end{equation}
\end{subequations}
\end{theorem}
In particular, this implies
\begin{corollary}
\begin{equation}\label{eq:7.14z}
\gamma_n:= \beta_{n+1}-\beta_n  = 2 + \frac{\xi_1}{2} n^{-1/2} + \hat \gamma_3 n^{-3/2} + \mathcal O (n^{-2}) \,.
\end{equation}
and
\begin{equation}\label{eq:7.14y}
\gamma_{n+1}-\gamma_{n} = - \frac 14 \xi_1 n^{-3/2} + \mathcal O (n^{-5/2})\,.
\end{equation}
\end{corollary}
This in particular proves for $n$ large that the sequence $\gamma_n$ is decreasing.
\\ 

{\it Proof of Theorem \ref{prop9.1}}\\
To prove \eqref{eq:7.13z},  we  implement our candidate for the expansion into \eqref{eq:sj1a} which reads at $\beta_n $:
$$
\beta_n  - (2n+1)= \sqrt{4 \lambda(n,\beta_n ) +1}\,.
$$
The $\hat \kappa_j$ are then obtained by recursion.\\

Let us more explicitly describe the first step.\\
We start from \eqref{eq:sj1a} at $\beta=\beta_n $ but we implement \eqref{eq:bh2}
\begin{equation}\label{eq:bh2impa}
\lambda(\beta_n )= \beta_n  \big(\Theta_0 - C_1(\beta_n ) ^{-1/2} + 3 C_1 \Theta_0^{1/2}(\frac 14 + C_0 )(\beta_n )^{-1}\big)  + \mathcal O (n^{-1/2} )\,.
\end{equation}
This leads to
\begin{equation*}
4 \lambda(\beta_n )+ 1= 4\beta_n  \big(\Theta_0 - C_1(\beta_n ) ^{-1/2} + C_2 (\beta_n )^{-1}\big)  + \mathcal O (n^{-1/2} )\,,
\end{equation*}
with
$$
C_2:= 3 C_1 \Theta_0^{1/2}(\frac 14 + C_0 )+ \frac 14 \,.
$$
We then get
$$
\sqrt{4 \lambda(\beta_n )+ 1} = 2 \Theta_0^{1/2}\beta_n^{1/2}  \big( 1 -  C_1\Theta_0^{-1} \beta_n^{-1/2} + C_2\Theta_0^{-1} \beta_n^{-1} + \mathcal O (n^{-3/2})   \big)^{\frac 12} 
$$
which implies
$$
\sqrt{4 \lambda(\beta_n )+ 1} = 2 \Theta_0^{1/2}\beta_n^{1/2}   -   C_1\Theta_0^{-1/2} + C_3\Theta_0^{-1} \beta_n^{-1/2} + \mathcal O (n^{-1})  \,,
$$
with $C_3$ computable.\\
This leads to
$$
\beta_n= 2n +1 + 2 \Theta_0^\frac 12\beta_n^{1/2}   -   C_1\Theta_0^{-1/2} + C_3\Theta_0^{-1} \beta_n^{-1/2} + \mathcal O (n^{-1}) \,.
$$
Implementing $\beta_n  =  2n  + \xi_1  n^\frac 12 +\hat \kappa_0  + \hat \kappa_1 n^{-1/2} + \mathcal O (n^{-1})$, we get the existence of explicit $\hat \kappa_0$ (as given in the statement)  and  $\hat \kappa_1$.
\qed

 From the complete expansion of $\beta_n$, we deduce
 \begin{theorem}
 There exist  sequences $\check \delta_j$, $\tilde \delta_j$ ($j\geq 1$) and $\check \lambda_j$ ($j\geq 3$) such that, for any $N$, we have 
\begin{equation}\label{eq:bh1z1}
\delta(n,\beta_n) = \delta_0-\frac 12 +\sum_{1\leq j \leq N}\check \delta_j \, n ^{-j/2} + {\mathcal O}(n ^{-\frac{N+1}{2}})\,,
\end{equation}
\begin{equation}\label{eq:bh1z2}
\delta(n+1,\beta_n) = \delta_0+ \frac 12 +\sum_{1\leq j \leq N}\tilde \delta_j \, n ^{-j/2} + {\mathcal O}(n ^{-\frac{N+1}{2}})\,,
\end{equation}
and 
\begin{equation}\label{eq:bh2z}
 \lambda(n,\beta_n )/\beta_n =  \Theta_0 + \lambda_1 \, n ^{-\frac{1}{2}} + \lambda_2(\delta_0-\frac 12)\, n ^{-1} +\sum_{j=3}^N \check \lambda_j \, n ^{-j/2} + {\mathcal O}(n ^{-\frac{N+1}{2}})\,.
\end{equation}
\end{theorem}
 \section{Strong diamagnetism revisited}
 \label{sec.strongdiamag}
 \subsection{Introduction}
 In the case of the disk, it was also proved in \cite{FH1} (Proposition 2.7) the following statement called ``strong diamagnetism".

\begin{theorem}
\label{DiamagDisc}~\\
Let $\mathbb D$ be the unit disk. Then the left- and right-hand derivatives $\lambda_{\pm}'(\beta)$ exist and satisfy
\begin{align}\label{eq:lb}
&\lambda_{+}'(\beta) \leq  \lambda_{-}'(\beta)\;, \nonumber\\
&\liminf_{\beta \rightarrow + \infty} \lambda_{+}'(\beta) \geq
\Theta_0 - \tfrac{3}{2}C_1 |\xi_0| >0\;.
\end{align}
In particular, $\beta \mapsto \lambda(\beta)$ is strictly increasing for large $\beta$.
\end{theorem}
Notice that the strong diamagnetism in dimension 3 is considered by Soeren  Fournais and Mikael Persson Sundqvist  in \cite{FP}.
Our  analysis of the Saint-James picture has given a more precise description by determining  the sequence $(\beta_n)$ where the left and right derivatives could differ. It is now  interesting to combine 
the information  given by the two approaches.
Our main  theorem is:
\begin{theorem}\label{th8.4}
We have, for suitable constants $\hat C_j$ and $\check C_j$  ($j\geq 1$),
\begin{equation}\label{eq:8.22}
\lambda'(n,\beta_n ) = \Theta_0
+ \frac 32 C_1 \Theta_0^{1/2}  + \sum_{j\geq 1} \hat C_j n^{-j/2}   + \mathcal O (n^{-\infty})\,.
\end{equation}
and
\begin{equation}\label{eq:8.23}
\lambda'(n+1,\beta_n ) = \Theta_0
- \frac 32 C_1 \Theta_0^{1/2}  + \sum_{j\geq 1} \check C_j n^{-j/2}   + \mathcal O (n^{-\infty})\,.
\end{equation}

\end{theorem}

We come back to $\lambda(n,\beta)$ since we are interested  in the asymptotics of $\lambda'(n,\beta)$.  We could start from
\begin{equation}
\lambda(n,\beta)  = \Theta_0 \beta + \lambda_1 \beta^{\frac{1}{2}} + \lambda_2 (\delta(n,\beta))) +\sum_{j=3}^N \lambda_j(\delta(n,\beta)) \beta^{1 -j/2} + {\mathcal O}(\beta^{1-\frac{N+1}{2}})\,,
\end{equation}
which is a consequence of \eqref{eq:trekant}, \eqref{eq:stjerne} and \eqref{eq:4.21}.\\
The only difficulty would be to justify the derivation with respect to $\beta$.  To avoid the difficulty we will use the Feynman-Hellmann formula.

\subsection{About $f_{n,\beta_n }$ and $f_{n+1,\beta_n }$ .}
\label{ss5.2}
Our aim is to prove:
\begin{proposition} For any $N>0$, 
\begin{align}
\beta_n^{-1/4} f_{n,\beta_n} (1)&=  u_0(0) + \sum_{j=1}^N \hat \kappa_j \beta_n^{-j/2}+\mathcal O (\beta^{-(N+1)/2})\,,\\
\beta_n^{-1/4} f_{n+1,\beta_n} (1)&=  u_0(0) + \sum_{j= 1}^N \check \kappa_j \beta_n^{-j/2}+\mathcal O (\beta^{-(N+1)/2})\,,
\end{align}
or
\begin{align}
\beta_n^{-1/4} f_{n,\beta_n} (1)&=  u_0(0) + \sum_{j= 1}^N \kappa_j  n^{-j/2}+\mathcal O (n^{-(N+1)/2})\,,\\
\beta_n^{-1/4} f_{n+1,\beta_n} (1)&=  u_0(0) + \sum_{j= 1}^N \tilde\kappa_j  n^{-j/2}+\mathcal O (n^{-(N+1)/2})\,
\end{align}
\end{proposition}
\begin{proof}
 This is actually a consequence of the comparison between $f_{n,\beta}$ and the quasi-mode constructed in \eqref{eq:4.24}. Taking the trace at $0$, we get
$$
u_{\delta(n,\beta),\beta} (0)= u_0(0) + \sum_{j=1}^N \hat c_j (\delta) \beta^{-j/2} + \mathcal O ( \beta^{-(N+1)/2}) \,.
$$
We finally observe that
$$
f_{n,\beta} (1)= \beta^{1/4} u_{\delta(n,\beta),\beta} (0) +\mathcal O (\beta^{-\infty})\,.
$$
The power of $\beta$ appears when doing the change of variable $t=\beta^{-1/2} \tau$. The error comes from the comparison between the disk and the annulus. We can take $\beta=\beta_n $, the comparison between $\beta_n $ and $n$ 
and  the complete expansion of $\delta (n,\beta_n)$ to achieve the proof of the expansions of $f_{n,\beta_n}(1)$. A similar reasoning, substituting $\delta(n+1,\beta_n)$ for $\delta(n,\beta_n)$, gives the expansions of $f_{n+1,\beta_n}(1)$.
\end{proof}
In particular, we get
 \begin{equation}\label{eq:exp}
  (\beta_n )^{-1/2} | f_{n,\beta_n }(1) |^2 = u(0)^2 + \sum_{j\geq 1}\hat K_j n^{-j/2} + \mathcal O(n^{-\infty})\,
  \end{equation}
  and
  \begin{equation}\label{eq:expBis}
  (\beta_n )^{-1/2} | f_{n+1,\beta_n }(1) |^2 = u(0)^2 + \sum_{j\geq 1}\check K_j n^{-j/2} + \mathcal O(n^{-\infty})\,.
  \end{equation}

\subsection{Feynman-Hellmann formula in the asymptotic limit}\label{ss5.3}
In order to support Conjecture \ref{conj:monotone}, let us present some additional results on the variation of $\beta\mapsto \lambda(n,\beta)$. Some of these computations appear already in \cite{FPS}. Combined with the asymptotic expansions
  given in Subsection \ref{ss5.2}, this gives a way to compute $\lambda'(n,\beta)$.

\begin{proposition}\label{lem:derivative} For $n\in\mathbb N$, for all $\beta>0$,
\begin{equation}\label{eq:derivative}
\lambda'(n,\beta)= \frac{\lambda(n,\beta)}{\beta} - \frac{1}{2\beta} \left( \lambda(n,\beta) - \big(n-\frac \beta  2\big)^2 \right) |f_{n,\beta}(1)|^2\,.
\end{equation}
\end{proposition}

\begin{proof} Using the notation of Sections \ref{s:curves} and \ref{s3}, we have, for $\beta>0$, 
\[\lambda(n,\beta)=\beta\,\eta(n,\beta)\]
and
\[\lambda'(n,\beta)=\eta(n,\beta)+\beta\,\eta'(n,\beta)\,.\]
Using Point (i) of Proposition \ref{prop:variation}, we get
\[\lambda'(n,\beta)=\eta(n,\beta)+\frac{\beta}{2}|g_{n,\beta}(\beta^{1/2})|^2\left(\left(\frac{n}{\beta^{1/2}}-\frac{\beta^{1/2}}{2}\right)^2-\eta(n,\beta)\right).\]
Since $g_{n,\beta}(\beta^{1/2})=\beta^{-1/2}f_{n,\beta}(1)$, we obtain the desired formula.
\end{proof}

\subsection{Application to the strong diamagnetism}\label{ss.diamagnetism}

We can use Formula \eqref{eq:derivative} to evaluate numerically the right derivative of $\beta\mapsto \lambda(\beta)$ at the points of intersection $\beta_n $. Indeed, it follows from Theorem~\ref{theorem:groundstate} that 
\begin{equation}\label{lambda+}
\lambda'_+(\beta_n )=\lambda'(n+1,\beta_n )\,,
\end{equation}
and
\begin{equation}\label{lambda-}
\lambda'_-(\beta_n )=\lambda'(n,\beta_n )\,.
\end{equation}
Assuming that Conjecture \ref{conj:monotone} is true, we should find $\lambda'_+(\beta_n )\ge 0$.

On the asymptotic side,
let us compute the term  which appears in \eqref{eq:derivative},
$$
 \frac{1}{2\beta_n } \left( \lambda(n,\beta_n ) - \big(n-\frac {\beta_n }{2}\big)^2 \right) \,.
 $$
 Using the improvement of  \eqref{eq:bh1}, 
 $$
 n- \frac {\beta_n }{2}=  \xi_0 (\beta_n )^{1/2} + \delta_0 - \frac 12 + \mathcal O (n^{-1/2})\,,
 $$
 and the expansion of $\lambda(n,\beta_n )$ given in \eqref{eq:bh2}, we obtain
 $$
 \lambda(n,\beta_n ) - \big(n-\frac {\beta_n }{2}\big)^2 = (- C_1 -  2 \xi_0 (\delta_0 - \frac 12)) (\beta_n )^{\frac 12} + \mathcal O (1)\,,
 $$
 and 
  $$
(\beta_n )^{-1}\Big( \lambda(n,\beta_n ) - \big(n-\frac {\beta_n }{2}\big)^2\Big) = \xi_0  (\beta_n )^{-\frac 12} + \mathcal O(n^{-1})\,.
 $$
 Using the asymptotics of $\beta_n $, $\lambda(\beta_n )$ and a complete expansion for $f_{n,\beta_n}$, we have the proof of \eqref {eq:8.22} in Theorem \ref{th8.4}. A similar proof, substituting $n+1$ for $n$ and using the expansion for  $f_{n+1,\beta_n}$, gives \eqref{eq:8.23}. Using \eqref{lambda-} and \eqref{lambda+} we deduce from Theorem \ref{th8.4}
 \begin{theorem}\label{th8.4strong}
We have 
\begin{equation}\label{eq:8.22a}
 \lambda'_- (\beta_n)= \Theta_0
+ \frac 32 C_1 \Theta_0^{1/2}  + \sum_{j\geq 1} \hat C_j  n^{-j/2}   + \mathcal O (n^{-\infty})\,,
\end{equation}
and 
\begin{equation}\label{eq:8.22b}
 \lambda'_+ (\beta_n)= \Theta_0
- \frac 32 C_1 \Theta_0^{1/2}  + \sum_{j\geq 1} \check C_j  n^{-j/2}   + \mathcal O (n^{-\infty})\,.
\end{equation}
\end{theorem}
In particular, we get
\begin{equation}\label{eq:zzz3}
\lim_{n\rightarrow +\infty} \lambda'(n+1,\beta_n )= \Theta_0- \frac 32 C_1 |\xi_0|\,,
\end{equation}
and
\begin{equation}\label{eq:zzz4}
\lim_{n\rightarrow +\infty} \lambda'(n,\beta_n )= \Theta_0 + \frac 32 C_1 |\xi_0|\,.
\end{equation}

Numerically, 
\begin{align*}
	 \Theta_0-\frac32C_1|\xi_0|&\sim 0.297350,\\
	 \Theta_0+\frac32C_1|\xi_0|&\sim 0.882863.
\end{align*}

 \section{Numerical study of the points of intersection}\label{s6}

Using the formulas of Subsection \ref{s2}, we look for an equation for $\nu$ as a function of $n$. We recall that for a pair $(x,\nu)$ corresponding to an intersection between $\beta\mapsto\eta(n,\beta)$ and $\beta\mapsto\eta(n+1,\beta)$, we have \eqref{eq:sj1aa}, that is
$$
x = (1-2 \nu +n+\frac 12) + \frac 12 \sqrt{ (3-4\nu)^2 + 8 (1-2\nu) n}\,,
$$
and \eqref{eqSys1}, that is 
$$
(n+1)(n-x)M(\nu,n+1,x)+2x\nu M(\nu+1,n+2,x)=0\, .
$$

By elimination of $x$, we get an equation of the form
\begin{equation}\label{eqImplicitSequence}
\Phi(\nu,n) =0\,,
\end{equation}
a priori satisfied  for $(\nu_n,n)$ with  $n\in \mathbb N$, but which could be analyzed for $n\in \mathbb R^+$.\\

Using \emph{Mathematica}, we determined numerically the points of intersection $(\beta_n, \eta_n^*)$ of the curves $\beta\mapsto\eta(n,\beta)$ and $\beta\mapsto\eta(n+1,\beta)$ by solving for $(\beta,\eta)$ 
in the non-linear system consisting of Equations \eqref{eqSys1} and \eqref{eqSys2}, with the integer $n$ ranging from $0$ to $400$. 
These points are displayed on Figure \ref{figIntersections} as black dots, for $n$ from 0 to $19$. A sample of the results for larger $n$ is presented in Table \ref{tabSyst}. This computation does not use Saint-James formula \eqref{eq:sj1aa}. The superscript is used to distinguish these numerical values from those obtained by a second method. We note that the sequence $(\eta_n^*)$ appears to be increasing.

\begin{table}[H]
\begin{center}
	\begin{tabular}[!p]{|c|l|l|}
	\hline
	\multicolumn{1}{|c|}{$n$ }&
	\multicolumn{1}{c}{$\beta_n^{(1)}$}&
	\multicolumn{1}{|c|}{ ${\eta_n^*}^{(1)}$}\\
	\hline
	0&  3.847538710016439&	0.46188467750410933\\
	1&	6.784689992385673&	0.490953836999826\\
	2&	9.495696565685895&	0.5057893193465876\\
	3&	12.091164794355297&	0.5152514126454681\\
	4&	14.613601105384173&	0.5219883372745205\\
	5&	17.08457097842645&	0.5271130898896494\\
	10&	28.989490930878333&	0.5418512305407657\\
	25&	62.88412636538398&	0.55750340973811\\
	50& 117.3339755112376&   0.5663294771262841\\
	100&	223.66235051600012&	0.5729419029706077\\
	200&	432.6371167436942&	0.5777978340023635\\
	300&	639.5318373766472&	0.5799955549150178\\
	400&	845.3470994716895&	0.5813189732301576\\
	\hline
	\end{tabular}
	
\end{center}
\caption{Solutions of the non-linear system}
\label{tabSyst}
\end{table}

\begin{figure}[H]
\begin{center}
	\includegraphics[width=.9\textwidth]{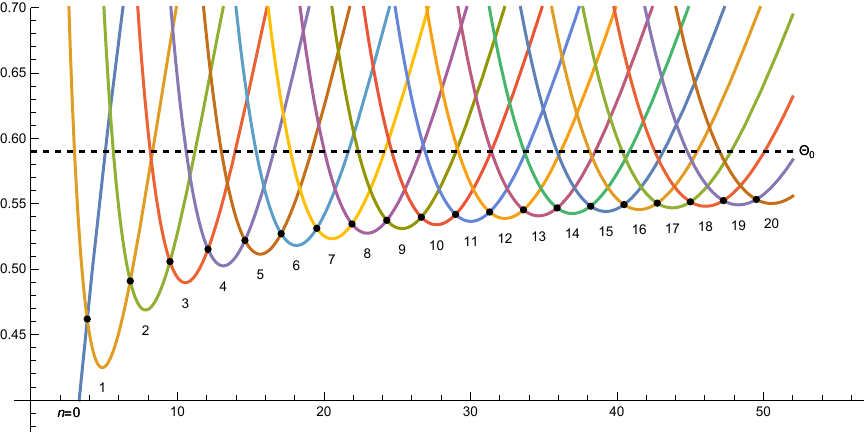}
	\caption{Zoom on the plot of $\beta\mapsto\eta(n,\beta)$ for $n$ from 0 to 20. The  black dashed line is the constant $\Theta_0$.  The black dots mark the points $(\beta_n^{(1)},{\eta_n^*}^{(1)})$.}
	\label{figIntersections}
\end{center}
\end{figure}

 Let us recall  that the sequence
$
\gamma_n:=\beta_{n+1}-\beta_n $
converges to $2$ and, more precisely, admits an asymptotic expansion of the form
\[\gamma_n=2+\sum_{j\ge1} \hat{\gamma}_j\,n^{-j/2}+O\left(n^{-\infty}\right)\]
(see Equation \eqref{eq:7.14z}). After applying four times the Richardson extrapolation\footnote{Assuming that $(y_n)$ is a sequence such that 
\[y_n=\ell+c\,n^{-k/2}+\mathcal O \left(n^{-(k+1)/2}\right)\]
for some $\ell,\,c \in\mathbb R$ and $k\in \mathbb{N}\setminus\{0\}$, the sequence
\[z_n:=\frac{2^{k/2}y_{2n}-y_n}{2^{k/2}-1}\]
satisfies
\[z_n=\ell+\mathcal O\left(n^{-(k+1)/2}\right).\]}, 
we obtain a sequence $(R_4\gamma_n)$ satisfying
\[R_4\gamma_n=2+\mathcal O \left(n^{-5/2}\right)\]
The new sequence therefore has the same limit as the original, with a much faster convergence, and can be used to check numerically the value of the limit. Let us note that we started with a finite number of terms from the sequence $(\gamma_n)$ and that each extrapolation divides by two the number of terms. We chose to extrapolate four times to balance the loss of terms with the gain in speed of convergence. The results are presented in Table \ref{tabDiff}. The blank cells in the rightmost column indicate terms of the sequence lost in the successive extrapolations. The term $\gamma_0$ is excluded from the process.

\begin{table}[h]
\begin{center}
	\begin{tabular}[!p]{|c|l|l|}
	\hline
	\multicolumn{1}{|c|}{$n$ }&
	\multicolumn{1}{c}{$\gamma_{n}$}&
	\multicolumn{1}{|c|}{ $ R_4\gamma_n$}\\
	\hline
	0&	2.9371512823692343&\\
	1&	2.7110065733002218&1.992309139244719\\
	2&	2.595468228669402&1.999385754276766\\
	3&	2.5224363110288763&2.0000780253629706\\
	4&	2.470969873042275&2.0001506529324526\\
	5&	2.4322005977636465&2.000134587773707\\
	6&	2.4016457058734133&2.000108698821863\\
	7&	2.376764086900799&2.0000863469016394\\
	8&	2.355993937178809&2.000068928763249\\	
	9&	2.338315624735216&2.000055641011863\\
	10&	2.3230315388610627&2.00004548367877\\
	24&	2.2159762440866757&2.0000068128960815\\
	50&	2.1516833824983337&\\
	100&	2.107942427892027&\\
	200&	2.076571962353569&\\
	300&	2.0625877291374763&\\
	399&	2.0542993010602686&\\
	\hline
	\end{tabular}
	
\end{center}
\caption{Sequence $\gamma_n=\beta_{n+1}-\beta_n$}
\label{tabDiff}
\end{table}

We computed the sequence $( {\eta_n^*}^{(2)})$, for the same values of $n$, by a second method: numerically solving Equation \eqref{eqImplicitSequence}, which is itself deduced from Saint-James formula \eqref{eq:sj1aa}. The results are presented in Table \ref{tabImpEq}, with the relative variation from the results of the first method, defined as
\[\varepsilon_n:=\frac{|{\eta^*}_n^{(2)}- {\eta^*}_n^{(1)}|}{{\eta^*}_n^{(1)}}.\]
The close agreement confirms Saint-James formula and suggests that our numerical computations are rather accurate.

\begin{table}[h]
\begin{center}
	\begin{tabular}{|c|l|l|}
	\hline
	\multicolumn{1}{|c|}{$n$ }&
	\multicolumn{1}{c}{${\eta_n^*}^{(2)}$}&
	\multicolumn{1}{|c|}{ $\varepsilon_n$}\\
	\hline
	0&	3.8475387100164355&	$10^{-15}$\\
	1&	6.784689992385671&	$2\times 10^{-15}$\\
	2&	9.495696565685904&	$10^{-15}$\\
	3&	12.091164794355297&	$3\times 10^{-15}$\\
	4&	14.613601105384173&	$3 \times 10^{-15}$\\
	5&	17.084570978426473&	$2\times 10^{-15}$\\
	10&	28.989490930878333&	$9\times 10^{-16}$\\
	25&	62.884126365384006&	$2\times 10^{-15}$\\
    50&	117.3339755112376&	$2\times 10^{-16}$\\
    100&	223.66235051600046&	$10^{-15}$\\
    200&	432.63711674369415&	$10^{-15}$\\
    300&	639.5318373766468&	$10^{-15}$\\
    400&	845.3470994716891&	$10^{-15}$\\ 
		\hline
	\end{tabular}
	
\end{center}
\caption{Solutions of the implicit equation, with the relative variation from Table \ref{tabSyst}}
\label{tabImpEq}

\end{table}

Finally, we present the numerical results concerning the left and right derivatives of $\lambda(\beta)$ at $\beta=\beta_n$, obtained with \emph{Mathematica}, in Table~\ref{table:derivatives}. 

To compare the limits \eqref{eq:zzz3} and \eqref{eq:zzz4} with our numerics, we proceed as we did for the sequence $(\gamma_n)$, applying four times the Richardson extrapolation to the numerical sequences. The results are displayed in Table \ref{table:derivatives}.

 All these numerical checks appear consistent with the asymptotic results of Sections \ref{s8} and \ref{sec.strongdiamag}.

\begin{table}[H]
\begin{center}
	\begin{tabular}[!p]{|c|c|c|c|c|c|}
	\hline
	\multicolumn{1}{|c|}{$n$ }&
	\multicolumn{1}{c|}{$\beta_n $}&
	\multicolumn{1}{c|}{ $\lambda'(n,\beta_n )$}&
	\multicolumn{1}{c|}{ $\lambda'(n+1,\beta_n )$}&
	\multicolumn{1}{c|}{ $R_4\lambda'(n,\beta_n )$}&
	\multicolumn{1}{c|}{ $R_4\lambda'(n+1,\beta_n )$}\\
	\hline
0&	3.84754&	0.884743&	0.144907& & \\
1&	6.78469&	0.880478&	0.178114&0.883010&	0.298037\\
2&	9.49570&	0.879482&	0.195254&0.882910&	0.297498\\
3&	12.0912&	0.879172&	0.206280&0.882882&   0.297401\\
4&	14.6136&	0.879085&	0.214181&0.882872&	0.297374\\
5&	17.0846&	0.879088&	0.220223&0.882868&	0.297363\\
6&	19.5168&	0.879130&	0.225048&0.882866&	0.297357\\
7&	21.9184&	0.879190&	0.229024&0.882861&	0.297355\\
8&	24.2952&	0.879257&	0.232376&0.882864&	0.297353\\
9&	26.6512&	0.879326&	0.235256&0.882864&	0.297352\\
10&	28.9895&	0.879394&	0.237767&0.882863&	0.297352\\
25&	62.8841&	0.880145&	0.256706&0.882863&	0.297350\\
50&	117.334&	0.880738&	0.267540&&\\
100&	223.662&	0.881254&	0.275737&&\\
200&	432.637&	0.881671&	0.281801&&\\
300&	639.532&	0.881870&	0.284558&&\\
400&	845.347&	0.881993&	0.286222&&\\
    
	\hline
	\end{tabular}
	\end{center}
	\caption{ Left and right derivatives at the points of intersection}
	\label{table:derivatives}
\end{table}

\end{document}